\documentclass{article}
\usepackage[utf8]{inputenc}
\usepackage{amssymb}
\usepackage{amsmath}
\usepackage{amsfonts}
\usepackage{hyperref}
\usepackage{graphicx}
\usepackage{tikz}
\usepackage{cancel}
\newtheorem{lemma}{Lemma}[section]
\newtheorem{theorem}[lemma]{Theorem}
\newtheorem{corollary}[lemma]{Corollary}
\newtheorem{definition}[lemma]{Definition}

\newtheorem{example}[lemma]{Example}
\newtheorem{construction}[lemma]{Construction}

\newenvironment{proof}[1][Proof]{\begin{trivlist}
\item[\hskip \labelsep {\bfseries #1}]}{\end{trivlist}}
\newcommand{\qed}{\nobreak \ifvmode \relax \else
      \ifdim\lastskip<1.5em \hskip-\lastskip
      \hskip1.5em plus0em minus0.5em \fi \nobreak
      \vrule height0.75em width0.5em depth0.25em\fi}
\title{Orientable sequences over non-binary alphabets}
\author{Abbas Alhakim, Chris J. Mitchell,
Janusz Szmidt, Peter R. Wild \\
aa145@aub.edu.lb, me@chrismitchell.net, \\
janusz.szmidt@tlen.pl, peterrwild@gmail.com}
\begin{document}

\maketitle

\section*{Abstract}

We describe new, simple, recursive methods of construction for \emph{orientable
sequences} over an arbitrary finite alphabet, i.e.\ periodic sequences in which any sub-sequence of $n$ consecutive elements occurs at most once in a period in either direction. In particular we establish how two variants of a generalised Lempel homomorphism can be used to recursively construct such sequences, generalising previous work on the binary case. We also derive an upper bound on the period of an orientable sequence.

\section{Introduction}

This paper is concerned with periodic infinite sequences with elements drawn from a finite alphabet with the property that for a positive integer $n$ (the \emph{order}) any sub-sequence of $n$ consecutive elements (an $n$-{\it tuple}) occurs just once in a period.  One important extremal class of such sequences are the de Bruijn sequences --- see, for example, \cite{Fredricksen82,Sawada23}. These sequences, sometimes referred to as shift register sequences (see Golomb, \cite{Golomb67}), have been very widely studied and have a range of applications, including in coding and cryptography.  One application of particular relevance here is that of \emph{position location}. This involves encoding such a sequence onto a linear surface that allows the location of any point on the surface by examining just $n$ consecutive entries of the sequence (see, for example, Burns and Mitchell \cite{Burns92,Burns93} and Petriu \cite{C35}). More recent work on the use of sequences for position location includes that of Szentandr{\'{a}}si et al.\ \cite{Szentandrasi12}, Bruckstein et al.\
\cite{Bruckstein12}, Berkowitz and Kopparty \cite{Berkowitz16}, and Chee et al.
\cite{Chee19,Chee20}.

We are particularly interested in the special case of \emph{orientable sequences} i.e.\ where, for a given positive integer $n$, any $n$-tuple of consecutive values occurs just once in a period \emph{in either direction} --- these sequences are a particular type of universal cycle --- see Chung et al.\ \cite{Chung92} and Jackson et al.\ \cite{Jackson09}.  Such sequences have position-location applications in cases where the reader of such a sequence wishes to determine both its position and its direction of travel.  The notion of an orientable sequence was first explored over 30 years ago \cite{Burns92,Burns93,Dai93}. Dai et al.\ \cite{Dai93} examined the binary case, and gave both a method of construction and a general upper bound for the period in this case. Much more recently, in 2022 an alternative method of construction for the binary case was described \cite{Mitchell22},
using the Lempel Homomorphism \cite{Lempel70}; this latter work has been further extended by Gabri\'{c} and Sawada \cite{Gabric24}.  However, as far as we are aware, no authors have previously addressed the issue of constructing, or bounding the periods of, orientable sequences over non-binary alphabets, a shortcoming we address here.

For mathematical convenience we consider the elements of a sequence to be elements of $\mathbb{Z}_q$ for an arbitrary integer $q>1$. The paper contains two main contributions relating to the existence of orientable sequences.  After establishing certain preliminary results, in Section~\ref{section-bound} we give a general upper bound for the period of an orientable sequence over an arbitrary alphabet. We then, in Section~\ref{section-construct}, show how generalisations of the Lempel Homomorphism \cite{Alhakim11,Ronse84} can be used to recursively construct orientable sequences over an arbitrary alphabet

\section{The basic notions}

For positive integers $n$ and $q$ greater than one, let $\mathbb{Z}_q^n$ be the set of all $q^n$
vectors of length $n$ with entries from the group $\mathbb{Z}_q$ of residues modulo $q$. A de Bruijn sequence \emph{of order $n$} with alphabet in $\mathbb{Z}_q$ is a periodic sequence that includes every
possible string of length $n$ (sometimes referred to as an $n$-tuple) precisely once as a subsequence of consecutive symbols in one period of
the sequence.

The order $n$ de Bruijn digraph, $B_n(q)$, is a directed graph with $\mathbb{Z}^n_q$ as its vertex
set and where, for any two vectors $\textbf{x} = (x_1, \dots , x_n)$ and $\textbf{y} = (y_1, \dots
, y_n), \ (\textbf{x}; \textbf{y})$ is an edge if and only if $y_i = x_{i+1}$ for every $i$ ($1\leq
i< n$). We then say that $\textbf{x}$ is a \emph{predecessor} of $\textbf{y}$ and $\textbf{y}$ is a
\emph{successor} of $\textbf{x}$. Evidently, every vertex has exactly $q$ successors and $q$
predecessors. Furthermore, two vertices are said to be \emph{conjugates} if they have the same set of successors.

A cycle in $B_n(q)$ is a path that starts and ends at the same vertex. It is said to be
\emph{vertex disjoint} if it does not visit any vertex more than once. Two cycles or two paths in
the digraph are vertex disjoint if they do not have a common vertex.

Following the notation of Lempel \cite{Lempel70}, a convenient representation of a vertex disjoint
cycle $(\textbf{x}^{(1)}; \dots ; \textbf{x}^{(l)})$ is the \emph{ring sequence} $[x^1, \dots , x^l
]$ of symbols from $\mathbb{Z}_q$ defined such that the $i$th vertex in the cycle starts with the
symbol $x^i$. A \emph{translate} of a word $\textbf{x} = (x_1, \dots , x_n)$ is a word $\textbf{x}+\lambda = (x_1 +
\lambda, \dots , x_n + \lambda)$ where $\lambda$ is any nonzero element in $\mathbb{Z}_q$ and addition is performed in $\mathbb{Z}_q$. We also define a translate of a cycle as the cycle
obtained by a translate of the ring sequence that defines this cycle.

%\fi

A cycle is said to be \emph{primitive} in $B_n(q)$ if it does not simultaneously contain a word and
any of its translates. A function $d : \mathbb{Z}^n_q \rightarrow Z_ q$ is said to be translation
invariant if $d(\textbf{x} + \lambda) = d(\textbf{x})$ for all $\textbf{x} \in \mathbb{Z}^n_q$ and all $\lambda \in
\mathbb{Z}_q$. The weight $w(\textbf{x})$ of a word or sequence $\textbf{x}$ is the sum of all elements in $\textbf{x}$ (not
taken modulo $q$). Similarly, the weight of a cycle is the weight of the ring sequence that
represents it. Obviously a de Bruijn sequence of order $n$ defines a Hamiltonian cycle in $B_n(q)$,
i.e.\ a cycle that visits each vertex exactly once and which we call a \emph{de Bruijn cycle}. We write $w_q(\mathbf{x})$ for $w(\mathbf{x}) \bmod{q}$.

\begin{definition}
Let $G_1$ and $G_2$ be two digraphs and $u$ and $v$ be arbitrary nodes in $G_1$. A function $H$
with domain $G_1$ and codomain $G_2$ is said to be a \emph{graph homomorphism} if ($H(u)$, $H(v)$)
is an edge in $G_2$ whenever ($u$, $v$) is an edge in $G_1$.
\end{definition}

For an integer $n > 1$ define a map $D : B_n(2) \rightarrow  B_{n-1}(2)$ by
\[ D(a_1, \dots , a_n) = (a_1 + a_2, a_2 + a_3, \dots , a_{n-1} + a_n) \]
where addition is modulo 2. This function defines a graph homomorphism and is known as Lempel's
D-morphism since it was first introduced by Lempel, \cite{Lempel70}.

We present a generalisation to nonbinary alphabets, \cite{Alhakim11}.

\begin{definition} \label{Lempel}
For a nonzero $\beta \in \mathbb{Z}_q$, we define a function $D_{\beta}$ from $B_n(q)$ to
$B_{n-1}(q)$ as follows. For $a = (a_1, \dots , a_n)$ and $b = (b_1, \dots , b_{n-1}), \
D_{\beta}(a) = b$ if and only if $b_i = d_{\beta}(a_i , a_{i+1})$ for $i = 1$ to $n-1$, where
$d_{\beta}(a_i , a_{i+1}) = \beta(a_{i+1} - a_i) \mod q.$
\end{definition}

Clearly $D_{\beta}$ is translation invariant.
It is also onto under a simple condition that $gcd(\beta, q) = 1$.

\begin{lemma}[Alhakim and Akinwande, \cite{Alhakim11}]
Every vertex in $B_{n-1}(q)$ has exactly $q$ inverse images in $B_n(q)$ by $D_{\beta}$ if and only
if $\beta$ is coprime with $q$.
\end{lemma}

\begin{theorem}[Alhakim and Akinwande, \cite{Alhakim11} and Ronse, \cite{Ronse84}]

Let $\Gamma$ be a vertex disjoint cycle of length $l$ in $B_{n-1}(q)$ and $D_{\beta}$ be as above,
where $\beta$ satisfies $gcd(\beta, q) = 1$. Then:
\begin{enumerate}
\item[(a)] $D_{\beta}$ is a graph homomorphism;
\item[(b)] The weight $W(\Gamma) \equiv \lambda \pmod q$, where $\lambda \neq 0$, if and only
    if $\Gamma$ is the image by $D_{\beta}$ of a vertex disjoint cycle of the form
\[ C = c \cdot (c + \beta^{-1}\lambda) \cdot (c + 2\beta^{-1}\lambda)  \dots (c + (r - 1)\beta^{-1}\lambda)  \]
obtained by concatenating $c$ with its translates, where $c = (x_1, \dots , x_l )$ is an
appropriate sequence of symbols of length $l$ and $r = q / gcd(\lambda, q)$; and
\item[(c)] $W(\Gamma) \equiv 0 \pmod q$ if and only if $\Gamma$ is the image by $D_{\beta}$ of
    $q$ primitive, vertex disjoint cycles of length $l$ in $B_n(q)$.
\end{enumerate}
\end{theorem}

\section{Orientable sequences}

\begin{definition}
We define an $n$-window sequence $S = (s_i)$ (see, for example, \cite{Mitchell96}) to be a periodic
sequence of period $m$ with the property that no $n$-tuple appears more than once in a period of
the sequence, i.e.\ with the property that if $\mathbf{s}_n(i) = \mathbf{s}_n(j)$ for some $i,j$, then $i \equiv j \pmod
m$, where $\mathbf{s}_n(i) = (s_i, s_{i+1}, \dots , s_{i+n-1})$.
\end{definition}

A de Bruijn sequence of order $n$ over the alphabet $\mathbb{Z}_q$ is then simply an $n$-window
sequence of period $q^n$ (i.e.\ of maximal period), and has the property that every possible
$n$-tuple appears once in a period. Since we are interested in tuples occurring either forwards or
backwards in a sequence we also introduce the notion of a \emph{reversed} tuple, so that if $\mathbf{u} = (u_0, u_1,
\ldots, u_{n-1})$ is a $q$-ary $n$-tuple, i.e.\ if $\mathbf{u} \in B_n(q)$, then $\mathbf{u}^R = (u_{n-1}, u_{n-2},
\ldots, u_0)$ is its reverse. If a tuple $\mathbf{u}$ satisfies $\mathbf{u} = \mathbf{u}^R$ then we say it is symmetric. A
\emph{translate} of a tuple involves switching $\mathbf{u} = (u_0,u_1, \ldots, u_{n-1}) \in B_n(q)$ to
$\mathbf{u}+\lambda  = (u_0 + \lambda, u_1 + \lambda, \ldots, u_{n-1} + \lambda)$, where $\lambda \in
\mathbb{Z}_q$. In a similar way, we refer to sequences being \emph{translates} if one can be obtained from
the other by the addition of a nonzero constant $\lambda$. We define a \emph{conjugate} of an $n$-tuple
to be a tuple obtained by adding $\lambda$ to the first element, i.e.\ if $\mathbf{u} = (u_0, u_1, \ldots, u_{n-1}) \in
B_n(q)$, then a conjugate $\widehat{\mathbf{u}}$ of $\mathbf{u}$ is an $n$-tuple $(u_0 + \lambda, u_1, \ldots, u_{n-1})$, where $\lambda \in \mathbb{Z}_q$.

Two $n$-window sequences $S = (s_i)$ and $T = (t_i)$ are said to be \emph{disjoint} if they do not share
an $n$-tuple, i.e.\ if $\mathbf{s}_n(i) \neq \mathbf{t}_n(j)$ for every $i, j$. An $n$-window sequence is said to be
\emph{primitive} if it is disjoint from each of its non-zero translates. We next give a well known result showing how two
disjoint $n$-window sequences can be \emph{joined} to create a single $n$-window sequence, if they
contain conjugate $n$-tuples.

\begin{lemma}
Suppose $S = (s_i)$ and $T = (t_i)$ are disjoint
$n$-window sequences of periods $l$ and $m$ respectively. Moreover
suppose $S$ and $T$ contain the conjugate $n$-tuples $\mathbf{u}$ and $\mathbf{v}$ at
positions $i$ and $j$, respectively. Then
\[  [s_0, s_1, \dots , s_{i+n-1}, t_{j+n}, t_{j+n+1}, \dots , t_{m-1},t_0, \dots , t_{j+n-1}, s_{i+n}, s_{i+n+1}, \dots , s_{l-1}]\]
is a ring sequence for an $n$-window sequence of period $l + m$.
\end{lemma}

\begin{definition}
An $n$-window sequence $S = (s_i)$ of period $m$ is said to be an orientable sequence of order $n$
(an $\mathcal{OS}_q(n)$) if, for any $i,j$, $\mathbf{s}_n(i) \neq \mathbf{s}_n(j)^R$.
\end{definition}

\begin{definition}
A pair of disjoint orientable sequences of
order $n, \ S = (s_i)$ and $S' = (s'_i)$, are said to be orientable disjoint
(or simply $o$-disjoint) if, for any $i, j, \ \mathbf{s}_n(i) \neq \mathbf{s}_n' (j)^R.$
\end{definition}

We extend the notation to allow the Lempel morphism $D_{\beta}$ to be applied to periodic sequences in the natural way, as we now describe. That is, $D_{\beta}$ (where $\beta\in \mathbb{Z}_q$) is the map from the set of periodic sequences to itself defined by
\[ D((s_i))= \{(t_i): t_j=\beta(s_{j+1}-s_{j}) \}. \]
The
image of a sequence of period $m$ will clearly have period dividing $m$. In the usual way we can
define $D_{\beta}^{-1}$ to be the \emph{inverse} of $D_{\beta}$, i.e.\ if $S$ is a periodic
sequence than $D_{\beta}^{-1}(S)$ is the set of all sequences $T$ with the property that
$D_{\beta}(T) = S$. The weight $w(S)$ of $S$ is the weight of the ring sequence corresponding to $S$. Similarly we write $w_q(S)$ for $w(S) \bmod{q}$.

We have the following result.

\begin{theorem}[Alhakim and Akinwande, Theorem 3.2 of \cite{Alhakim11}] \label{D-beta}
Let the sequence $S$ be a $n$-window sequence with period $m$. If $gcd(w(S), q) = g$ then
$D^{-1}_{\beta}(S)$ consists of $g$ disjoint cycles each of period $\frac{q}{g}m$. In particular:
\begin{itemize}
    \item if $g = q$ (i.e. $w(S) = 0 \mod q$) then $D^{-1}_{\beta}(S)$ consists of $q$ disjoint
        cycles that are translates of each other; and
    \item if $g = 1$ then $D^{-1}_{\beta}(S)$ is one cycle of length $qm$.
\end{itemize}
\end{theorem}

Note that a further interesting extension of the Lempel Homomorphism has been described in recent work by Nellore and Ward \cite{Nellore22}.

\section{An upper bound} \label{section-bound}

We first introduce a special type of symmetry for $q$-ary $n$-tuples.

\begin{definition}
An $n$-tuple $\mathbf{u} = (u_0, u_1, \dots , u_{n-1}), \ u_i \in \mathbb{Z}_q \ (0 \leq i \leq n - 1)$, is $m$-symmetric for some $m \leq n$ if and only if $u_i = u_{m-1-i}$ for every $i \ (0 \leq i \leq m - 1)$.
\end{definition}

An $n$-tuple is simply said to be symmetric if it is $n$-symmetric. We also need the notions of uniformity and alternating.

\begin{definition}
An $n$-tuple $\mathbf{u}=(u_0,u_1,\ldots,u_{n-1})$, $u_i\in\mathbb{Z}_q$ ($0\leq i\leq n-1$), is
\emph{uniform} if and only if $u_i=c$ for every $i$ ($0\leq i\leq n-1$) for some
$c\in\mathbb{Z}_q$. An $n$-tuple $\mathbf{u}=(u_0,u_1,\ldots,u_{n-1})$, $u_i\in\mathbb{Z}_q$ ($0\leq i\leq n-1$), is \emph{alternating} if and only if $u_0=u_{2i}$ and $u_1=u_{2i+1}$ for every $i$ ($0\leq i\leq \lfloor(n-1)/2\rfloor$), where $u_0\not=u_1$.
\end{definition}

We can then state the following elementary results.

\begin{lemma} \label{lemma_symmetry}
If $n\geq2$ and $\mathbf{u}=(u_0,u_1,\ldots,u_{n-1})$ is a $q$-ary $n$-tuple that is both symmetric
and $(n-1)$-symmetric, then $\mathbf{u}$ is uniform.
\end{lemma}

\begin{proof} Choose any $i$, $0\leq i\leq n-2$. Then, by symmetry we know that $u_{i+1}=u_{n-2-i}$ (observing that $i\leq n-2$).  Also, by $(n-1)$-symmetry we know that $u_i=u_{n-2-i}$, and hence we have
$u_i=u_{i+1}$.  The result follows. \qed
\end{proof}

\begin{lemma} \label{lemma_symmetry_2}
If $n\geq2$ and $\mathbf{u}=(u_0,u_1,\ldots,u_{n-1})$ is a $q$-ary $n$-tuple that is both symmetric
and $(n-2)$-symmetric then:
\begin{itemize}
\item if $n$ is even then $\mathbf{u}$ is uniform, and
\item if $n$ is odd then $\mathbf{u}$ is either uniform or alternating.
\end{itemize}
\end{lemma}

\begin{proof}
Choose any $i$, $0\leq i\leq n-3$. Then, by symmetry we know that $u_{i+2}=u_{n-3-i}$ (observing that $i\leq n-3$).  Also, by $(n-2)$-symmetry we know that $u_i=u_{n-3-i}$, and hence we have
$u_i=u_{i+2}$.  Hence $\mathbf{u}$ is either uniform or alternating.  To complete the proof, suppose $n$ is even; let $n=2m$. By symmetry we know that $u_{m+1}=u_m$, and hence $\mathbf{u}$ is uniform. \qed
\end{proof}

The following definition leads to a simple upper bound on the period of an $\mathcal{OS}_q(n)$.

\begin{definition}
Let $N_q(n)$ be the set of all non-symmetric $q$-ary $n$-tuples.
\end{definition}

Clearly, if an $n$-tuple occurs in an $\mathcal{OS}_q(n)$ then it must belong to $N_q(n)$; moreover
it is also immediate that $|N_q(n)|=q^n-q^{\lceil n/2\rceil}$.  Observing that all the tuples in
 $\mathcal{OS}_q(n)$ and its reverse must be distinct, this immediately give the following
well-known result.

\begin{lemma}[\cite{Burns93}, Lemma 15] The period of an $\mathcal{OS}_q(n)$ is at most $(q^n-q^{\lceil
n/2\rceil})/2$.
\end{lemma}

As a first step towards establishing our bound we need to define a special set of $n$-tuples, as
follows.

\begin{definition}
Suppose $n\geq 2$, and that $\mathbf{v}=(v_0,v_1,\ldots,v_{n-r-1})$ is a $q$-ary $(n-r)$-tuple ($r\geq 1$).  Then let $L_n(\mathbf{v})$ be the following set of $q$-ary $n$-tuples:
\[  L_n(\mathbf{v}) = \{ \mathbf{u}=(u_0,u_1,\ldots,u_{n-1}):~
u_i=v_i,~~0\leq i\leq n-r-1 \}. \]
\end{definition}

That is, $L_n(\mathbf{v})$ is simply the set of $n$-tuples whose first $n-r-1$ entries equal
$\mathbf{v}$.  Clearly, for fixed $r$, the sets $L_n(\mathbf{v})$ for all $(n-r)$-tuples $\mathbf{v}$ are disjoint. We have the following simple result.

\begin{lemma} \label{lemma_disjoint}
Suppose $\mathbf{v}$ and $\mathbf{w}$ are symmetric tuples of lengths $n-1$ and $n-2$, respectively, and they are not both uniform.  Then
\[ L_n(\mathbf{v}) \cap L_n(\mathbf{w}) = \emptyset. \]
\end{lemma}

\begin{proof}
Suppose $\mathbf{u}\in L_n(\mathbf{v}) \cap L_n(\mathbf{w})$; this implies that $\mathbf{v}\in L_{n-1}(\mathbf{w})$. Hence $\mathbf{v}$ is both $(n-1)$-symmetric and $(n-2)$-symmetric; so (from Lemma~\ref{lemma_symmetry}), $\mathbf{v}$ and $\mathbf{w}$ are both uniform, giving the desired contradiction.
\qed
\end{proof}

We are particularly interested in how the sets $L_n(\mathbf{v})$ intersect with the sets of $n$-tuples occurring in either $S$ or $S^R$, when $S$ is an $\mathcal{OS}_q(n)$ and $\mathbf{v}$ is symmetric.  To this end we make the following definition.

\begin{definition}
Suppose $n\geq 2$, $r\geq 1$, $S=(s_i)$ is an $\mathcal{OS}_q(n)$, and $\mathbf{v}=(v_0,v_1,\ldots,v_{n-r-1})$ is a $q$-ary $(n-r)$-tuple.  Then let
\[ L_S(\mathbf{v})=\{\mathbf{u}\in L_n(\mathbf{v}):~
 \mathbf{u}~\mbox{appears in}~S~\mbox{or}~S^R \}.\]
\end{definition}

We can now state the first result towards deriving our bound.

\begin{lemma} \label{lemma_even}
Suppose $n\geq 2$, $r\geq 1$, $S=(s_i)$ is an $\mathcal{OS}_q(n)$, and $\mathbf{v}=(v_0,v_1,\ldots,v_{n-r-1})$ is a $q$-ary symmetric $(n-r)$-tuple.  Then $|L_S(\mathbf{v})|$ is even.
\end{lemma}

\begin{proof}
Suppose $L_S(\mathbf{v})$ is non-empty.  Then, by definition, and (without loss of generality)
assuming that an element of $L_n(\mathbf{v})$ occurs in $S$ (as opposed to $S^R$), we know that
$v_i=s_{j+i}$ for some $j$ ($0\leq i\leq n-r-1$).  Since $\mathbf{v}$ is symmetric, it follows
immediately that $v_i=s_{j+n-r-1-i}$ ($0\leq i\leq n-r-1$).  That is, for each occurrence of an element
of $L_n(\mathbf{v})$ in $S$, there is an occurrence of a (necessarily distinct) element of
$L_n(\mathbf{v})$ in $S^R$, and vice versa.  The result follows. \qed
\end{proof}

That is, if $|L_n(\mathbf{v})|$ is odd, this shows that $S$ and $S^R$ combined must omit at least one
of the $n$-tuples in $L_n(\mathbf{v})$.  We can now state our main result.

Observe that, although the theorem below applies in the case $q=2$, the bound is much weaker than
the bound of Dai et al.\ \cite{Dai93} which is specific to the binary case. This latter bound uses
arguments that only apply for $q=2$. The fact that $q=2$ is a special case can be seen by observing
that, unlike the case for larger $q$, no string of $n-2$ consecutive zeros or ones can occur in an
$\mathcal{OS}_2(n)$. The result can be seen as a generalisation of \cite{Dai93}, Theorem 2.2.

\begin{theorem} \label{thm_bounds}
Suppose that $S=(s_i)$ is an $\mathcal{OS}_q(n)$ ($q\geq2$, $n\geq 2$). Then the period of $S$ is
at most
\begin{eqnarray*}
(q^n - q^{\lceil n/2\rceil} - q^{\lceil (n-1)/2\rceil} + q)/2 & \mbox{if $q$ is odd}, \\
(q^n - q^{\lceil n/2\rceil} - q)/2 & \mbox{if $q$ is even}.
\end{eqnarray*}
Further, if $q$ is odd and $n\geq 6$ then the period of $S$ is at most
\begin{eqnarray*}
(q^n - 2q^{n/2} - q^{(n-2)/2} + 2q)/2 & \mbox{if $n$ is even}, \\
(q^n - q^{(n+1)/2} - 2q^{(n-1)/2} + q + q^2)/2 & \mbox{if $n$ is odd}.
\end{eqnarray*}
\end{theorem}

\begin{proof}

The proof involves considering the set $N_q(n)$ of non-symmetric $q$-ary $n$-tuples, and showing
that the tuples in certain disjoint subsets of this set cannot all occur in an $\mathcal{OS}_q(n)$.
We divide our discussion into two cases depending on the parity of $q$.

\begin{itemize}
\item First suppose that $q$ is odd.  Suppose also that $\mathbf{v}$ is a symmetric non-uniform
    $(n-1)$-tuple; there are clearly $q^{\lceil (n-1)/2\rceil}-q$ such tuples.

Next consider $L_n(\mathbf{v})$, the set of $q$-ary $n$-tuples whose first $n-1$ values are the same
as those of $\mathbf{v}$; clearly $|L_n(\mathbf{v})|=q$.  Since $\mathbf{v}$ is symmetric, all
elements of $L_n(\mathbf{v})$ are $(n-1)$-symmetric.  Hence, by Lemma~\ref{lemma_symmetry}, and since
$\mathbf{v}$ is not uniform, no element of $L_n(\mathbf{v})$ can be symmetric, i.e.\
$L(\mathbf{v})\subseteq N_q(n)$.
Now since $q$ is odd, Lemma~\ref{lemma_even} implies that at least one element of
$L_n(\mathbf{v})$ cannot be contained in $L_S(\mathbf{v})$.  That is, for each of the $q^{\lceil
(n-1)/2\rceil}-q$ symmetric non-uniform $q$-ary $(n-1)$-tuples $\mathbf{v}$, at least one element of
$L_n(\mathbf{v})$ cannot be contained in $L_S(\mathbf{v})$.  Observing that the sets
$L_n(\mathbf{v})$ are all disjoint and --- as noted above --- $L_n(\mathbf{v})\subseteq N_q(n)$ for every symmetric non-uniform $(n-1)$-tuple $\mathbf{v}$, it follows that the period of $S$ is bounded above by $(|N_q(n)|-(q^{\lceil (n-1)/2\rceil}-q))/2$. The bound for $q$ odd follows.

\item Now suppose $q$ is even.  Let $\mathbf{v}$ be a uniform $(n-1)$-tuple; there are clearly
    $q$ such tuples. As above, we know that $|L_n(\mathbf{v})|=q$.  Precisely one of the elements
    of $L_n(\mathbf{v})$ is symmetric, namely the relevant uniform $n$-tuple, and hence
    $|L_n(\mathbf{v})\cap N_q(n)|=q-1$. Now, since $q$ is even, Lemma~\ref{lemma_even} implies
    that at least one element of $L_n(\mathbf{v})\cap N_q(n)$ cannot be contained in
    $L_S(\mathbf{v})$. As previously, the period of $S$ is thus bounded above by $(|N_q(n)|-q)/2$, and the result for $q$ even follows.
\end{itemize}

To establish the slightly stronger bounds for $n\geq 4$ ($q$ odd) we need to consider the cases $n$ even and odd separately.
\begin{itemize}
\item First suppose $n$ is even. Suppose $\mathbf{w}$ is a symmetric non-uniform $(n-2)$-tuple; there are $q^{(n-2)/2}-q$ such tuples. As previously, consider $L_n(\mathbf{w})$, the set of $q$-ary $n$-tuples whose first $n-2$ values are the same as those of $\mathbf{w}$; clearly $|L_n(\mathbf{w})|=q^2$.  Since $\mathbf{w}$ is symmetric, all elements of $L_n(\mathbf{w})$ are $(n-2)$-symmetric.  Hence, by Lemma~\ref{lemma_symmetry_2}, and since $\mathbf{w}$ is not uniform, no element of $L_n(\mathbf{w})$ can be symmetric, i.e.\ $L(\mathbf{w})\subseteq N_q(n)$. Just as previously, since $q$ (and hence $q^2$) is odd, Lemma~\ref{lemma_even} implies that at least one element of $L_n(\mathbf{w})$ cannot be contained in $L_S(\mathbf{w})$.  That is, for each of the $q^{(n-2)/2}-q$ symmetric non-uniform $q$-ary $(n-2)$-tuples $\mathbf{w}$, at least one element of $L_n(\mathbf{w})$ cannot be contained in $L_S(\mathbf{v})$. Finally, observing that (a) the sets $L_n(\mathbf{w})$ are all disjoint, (b) the sets $L_n(\mathbf{w})$ are disjoint from the sets $L_n(\mathbf{v})$ (for symmetric non-uniform $(n-1)$-tuples $v$), and (c) $L_n(\mathbf{w})\subseteq N_q(n)$ for every $\mathbf{w}$, it follows that the length of $S$ is bounded above by
\begin{eqnarray*}
(|N_q(n)| - (q^{\lceil (n-1)/2\rceil}-q) - (q^{(n-2)/2}-q))/2 \\
= (q^n-q^{n/2}) - (q^{n/2}-q) - (q^{(n-2)/2}-q))/2  \mbox{~~~(since $n$ is even).}
\end{eqnarray*}
and the bound follows.
\item Now suppose $n$ is odd. Suppose $\mathbf{w}$ is a symmetric non-uniform non-alternating $(n-2)$-tuple; there are $q^{(n-1)/2}-q^2$ such tuples (since $n$ is odd). Using precisely the same argument as the $n$ even case, it follows that the length of $S$ is bounded above by
\begin{eqnarray*}
(|N_q(n)| - (q^{\lceil (n-1)/2\rceil}-q) - (q^{(n-1)/2}-q^2))/2 \\
= (q^n-q^{(n+1)/2}) - (q^{(n-1)/2}-q) - (q^{(n-1)/2}-q^2))/2  \mbox{~~~(since $n$ is odd).}
\end{eqnarray*}
and the desired result follows. \qed
\end{itemize}
\end{proof}

We conclude by tabulating the values of the bounds of the above (from Theorem~\ref{thm_bounds}) for
small $q$ and $n$.
\begin{table}[htb]
\caption{Bounds on the period of an $\mathcal{OS}_q(n)$ (from Theorem~\ref{thm_bounds})}
\label{table_bounds}
\begin{center}
\begin{tabular}{ccccc} \hline
Order  & $q=2$ & $q=3$ & $q=4$ & $q=5$ \\ \hline
$n=2$  &   0   &    3   &     4 &    10 \\
$n=3$  &   1   &    9   &    22 &    50 \\
$n=4$  &   5   &   33   &   118 &   290 \\
$n=5$  &  11   &  105   &   478 &  1490 \\
$n=6$  &  27   &  336   &  2014 &  7680 \\
$n=7$  &  55   & 1032   &  8062 & 38640 \\  \hline
\end{tabular}
\end{center}
\end{table}

\section{Some simple examples}

For small $n$ the above bounds appear to be quite tight.  To illustrate this we give some simple
examples.

\begin{example}  \label{example-q3}
We first examine $q=3$.  It is simple to see that [012] is an $\mathcal{OS}_3(2)$ of period 3. From
Table~\ref{table_bounds} it follows that this has optimally large period and in this case the bound
of Theorem~\ref{thm_bounds} is tight.

Applying the inverse of Homomorphism A (defined in the next section) to [012] gives the following
pair of $\mathcal{OS}_3(3)$s: [001122] and [210].  If we reverse the second sequence to get [012],
it is simple to join them (using the 2-tuple 01 which occurs in both sequences) to get the
following $\mathcal{OS}_3(3)$: [001201122]. This has period 9, and shows that the bound is again
tight in this case.
\end{example}

\begin{example}  \label{example-q5}
We next examine $q=5$.  It is simple to see that [0123402413] is an $\mathcal{OS}_5(2)$ of period
10. From Table~\ref{table_bounds} it follows that this has optimally large period and in this case
the bound of Theorem~\ref{thm_bounds} is again tight.
\end{example}

Examining Examples~\ref{example-q3} and \ref{example-q5} motivates the following observation.

\begin{construction} \label{construction-simple}
Let $p$ be prime. For $i=0,\dots,\frac{(p-3)}{2}$ and $j=0,\dots,p-1$ let $\ell=ip+j$ and put
$s_{\ell}=ij \pmod p$. Consider the periodic $p$-ary sequence $S$ of period $\frac{p(p-1)}{2}$
having ring sequence $[s_0,\dots,s_{\frac{p(p-1)}{2}-1}]$.
\end{construction}

\begin{theorem}
Let $p$ be prime and suppose $S$ is as in Construction~\ref{construction-simple}. Then $S$ is an
$O_p(2)$ of maximum period.
\end{theorem}

\begin{proof} Now $s_{t}=0$ for $t=ip$; $i=0,\dots,\frac{(p-3)}{2}$, and we have $s_{t+1}=i+1 \pmod
p$ and $s_{t-1}=p-i \pmod p$ (noting that we identify $s_{-1}$ with $s_{\frac{p(p-1)}{2}-1}$). Thus
each pair $(0,i)$, $i$ a non-zero residue modulo $p$, occurs exactly once as a $2$-tuple or reverse
$2$-tuple in a cycle of $S$.

Let $\alpha,\beta$ be distinct non-zero residues modulo $p$. Then $\beta -\alpha \in
\{1,\dots,\frac{(p-1)}{2}\}$ or $\alpha-\beta \in \{1,\dots,\frac{(p-1)}{2}\}$.

In the former case $\frac{\beta}{\beta -\alpha}=\frac{\alpha}{\beta-\alpha}+1$ and
$(\alpha,\beta)=(s_{\ell},s_{\ell+1})$ where $\ell \pmod p=\frac{\alpha}{\beta-\alpha} $ and  $0\le
\ell \le \frac{p(p-1)}{2}-1$. In the latter case $\frac{\beta}{\beta
-\alpha}=\frac{\alpha}{\beta-\alpha}-1$ and $(\alpha,\beta)$ is the reverse of
$(s_{\ell-1},s_{\ell})$ where $\ell \pmod p=\frac{\alpha}{\beta-\alpha} $ and  $0\le \ell \le
\frac{p(p-1)}{2}-1$.

It follows that each non-symmetric $2$-tuple occurs either as a $2$-tuple or a reverse $2$-tuple in
a cycle of $S$ and hence occurs exactly once. Thus $S$ is an $O_p(2)$. That the period is maximum
follows from Theorem~\ref{thm_bounds}. \qed
\end{proof}

The situation for even $q$ is less clear, as the next example shows.

\begin{example}  \label{example-q4}
Suppose $q=4$.  It is simple to see that [0123] is an $\mathcal{OS}_4(2)$ of period 4. From
Table~\ref{table_bounds} it follows that this has optimally large period and in this case the bound
of Theorem~\ref{thm_bounds} is tight.

Applying the inverse of Homomorphism A (defined in the next section) to this sequence gives the
following pair of $\mathcal{OS}_4(3)$s: [00112233] and [13203102].  Reversing the second sequence
(to get [20130231]) and joining them using the shared 2-tuple 01 yields the following
$\mathcal{OS}_4(3)$: [0013023120112233]. This has period 16; this is clearly significantly less
than 22, the bound in Table~\ref{table_bounds}.  Whether or not an $\mathcal{OS}_4(3)$ with period greater
than 16 exists is an open question that should be readily amenable to computer search.
\end{example}

\section{Constructing orientable sequences} \label{section-construct}

\subsection{Getting started}

We now show how the generalised Lempel Homomorphism can be used to construct orientable
sequences over an arbitrary alphabet. We first need to establish some more terminology.

\begin{definition}
Let $S =(s_0, s_1, \dots , s_{m-1})$ be a sequence. We define the {\em alternating sign}
weight of $S$ to be
\[ w^A_q(S) = \sum_{i=0}^{m-1} (-1)^{m-1-i}s_i.\]
%where $w^A_q(S) \in \mathbb{Z}_q$.
\end{definition}

\begin{definition}
An $n$-tuple $\mathbf{u}=(u_0,u_1,\ldots,u_{n-1})$, $u_i\in\mathbb{Z}_q$ ($0\leq i\leq n-1$), is of
\emph{alternating sign form} if and only if $u_{i+1}=-u_i$ for every $i$
($0\leq i\leq n-2$) and we write $\underline{c}^n$ for such an $n$-tuple if $u_0=c$.
\end{definition}

\begin{definition}
The \emph{negative} of a $q$-ary $n$-tuple $\mathbf{u}=(u_0,\dots,u_{n-1})$ is the $n$-tuple
%$\underline{\mathbf{u}}=1^n \oplus (q-1-u_0,\dots,q-1-u_{n-1})$.
$-\mathbf{u}=(-u_0,\dots,-u_{n-1})$.
\end{definition}

\begin{definition}
An $n$-window sequence $S=(s_i)$ of period $m$ is said to be a \emph{negative orientable sequence
of order $n$} (an $\mathcal{NOS}(n)$) if, for any $i,j$,
$\mathbf{s}_n(i)\not=-{\mathbf{s}_n(j)}^R$.
\end{definition}

\begin{definition}
A pair of disjoint orientable sequences of order $n$, $S=(s_i)$ and $S'=(s'_i)$, are said to be
\emph{negative orientable-disjoint} (or simply \emph{no-disjoint}) if, for any $i,j$,
$\mathbf{s}_n(i)\not=-{\mathbf{s'}_n(j)}^R$.
\end{definition}

\subsection{Two variants of the Lempel homomorphism}

We next introduce a variant of the generalised Lempel Homomorphism.

\begin{definition}[Lempel \cite{Lempel70}]
The mapping $A:\mathbb{Z}_q^n\rightarrow\mathbb{Z}_q^{n-1}$ is as follows. If
$\mathbf{u}=(u_0,u_1,\ldots,u_{n-1})\in\mathbb{Z}_q^n$ then
\[ A(\mathbf{u}) = (u_1+u_0,u_2+u_1,\dots,u_{n-1}+u_{n-2}) \in\mathbb{Z}_q^{n-1}. \]
\end{definition}

%We consider Definition~\ref{Lempel} with $\beta = 1$ and denote it $D$.
In the binary case $D_{\beta}=A$.
We extend the notation to allow $A$ and $D_{\beta}$ to be applied to $q$-ary sequences in the natural way.
We next show that a similar approach to that given in \cite{Mitchell22} can be used to construct
orientable $q$-ary sequences of order $n+1$ from one of order $n$. We also consider a special type
of $q$-ary orientable sequence and make the following definitions.

\begin{definition}
An orientable sequence $S=(s_i)$ of order $n$ is said to be \emph{special} if, for any $i,j$,
$\mathbf{{s}}_n(i)\not=-{\mathbf{s}}_n(j)^R$, i.e. it is also negative orientable.
\end{definition}

\begin{definition}
Two o-disjoint orientable sequences $S=(s_i)$  and $T=(t_i)$ are special-orientable-disjoint
(or simply \emph{special-o-disjoint}) if, for any $i,j$,
$\mathbf{s}_n(i)\not=-{\mathbf{t}}_n(j)^R$, i.e. they are also negative
orientable-disjoint.
\end{definition}

We consider first the homomorphism $D_{\beta}$.  We will take $\beta=1$ and write $D$ for $D_1$.  For all $\beta$ with $gcd(\beta,q)=1$, the next result follows in essentially the same way.

\begin{theorem}  \label{thm Lempel-orientable}
Suppose $S=(s_i)$ is an orientable sequence of order $n$ and period $m$. If $w_q(S)$ has additive
order $h$ as a residue modulo $q$ then $D^{-1}(S)$ consists of $h$ shifts of each of $q/h$ mutually
no-disjoint negative orientable sequences of order $n+1$ and period $hm$ which are translates of
one another.
\end{theorem}

\begin{proof}
Let $w_q(S)$ have order $h$. Let $T=(t_i)$ and $T^{\prime}=(t_i^{\prime})$ be two elements of $D^{-1}(S)$. Then
%\[ t_i=t_0\oplus \bigoplus_{j=0}^{i-1}s_j \
%\mbox{\rm and} \
%t_i^{\prime}=t_0^{\prime}\oplus \bigoplus_{j=0}^{i-1}s_j. \]
\[ t_i=t_0+ \sum_{j=0}^{i-1}s_j \
\mbox{\rm and} \
t_i^{\prime}=t_0^{\prime}+\sum_{j=0}^{i-1}s_j. \]
It is clear that $T$ and $T^{\prime}$ are translates of one another, and there are $q$ such translates in $D^{-1}(S)$.

Note that, for a positive integer $\alpha$, $t_{i+\alpha m}=t_i + \alpha w_q(S)$, so that $t_{i+hm}=t_i$ for all $i$ and $h$ is the smallest positive integer with this property. Moreover, if $t_0^{\prime}=t_0 + \alpha w_q(S)$ then $t_i^{\prime}=t_{i+\alpha m}$ for all $i$, so that $T^{\prime}$ is a shift of $T$ by $\alpha m$ places. Also, if, for some positive integer $\ell$, $t_{i+\ell}=t_i$ for all $i$, then $s_{i+\ell}=s_i$ for all $i$ so that $m$ divides $\ell$. Hence $T$ has period $hm$, and $D^{-1}(S)$ contains $h$ shifts of $T$ by various multiples of $m$.

Let $j_1,j_2$ be non-negative integers and suppose that $i_1,i_2 \in \{0,1,\dots,m-1\}$ with $j_1 \equiv i_1 \pmod m$ and $j_2\equiv i_2 \pmod m$.

If $\mathbf{t}_{n+1}(j_2)=\mathbf{t}_{n+1}(j_1)$ then $D(\mathbf{t}_{n+1}(j_2))=\mathbf{s}_n(i_2)=D(\mathbf{t}_{n+1}(j_1))=\mathbf{s}_n(i_1))$ so that $i_2=i_1$. Without loss of generality suppose that $j_2=j_1+\alpha m$ for some positive integer $\alpha$. Since we then have $t_{j_2}=t_{j_1}+\alpha w_q(S)=t_{j_1}$ and since $w_q(S)$ has order $h$, necessarily $h$ divides $\alpha$. It follows that $j_2=j_1+\ell hm$ for some integer $\ell$ and $T$ is an $(n+1)$-window sequence.

%If $\mathbf{t}_{n+1}(j_2)=\underline{\mathbf{t}}^R_{n+1}(j_1)$ then $D(\mathbf{t}_{n+1}
%(j_1))=\mathbf{s}_n(i_2)=D(\underline{\mathbf{t}}^R_{n+1}(j_1))=\mathbf{s}^R_n(i_1))$.
%This is impossible and so $T$ is a negative orientable sequence.

If $\mathbf{t}_{n+1}(j_2)=-{\mathbf{t}}^R_{n+1}(j_1)$ then $D(\mathbf{t}_{n+1}(j_2))=\mathbf{s}_n(i_2)=D(-{\mathbf{t}}^R_{n+1}(j_1))=\mathbf{s}^R_n(i_1)$. This is impossible and so $T$ is a negative orientable sequence.

If $\mathbf{t}^{\prime}_{n+1}(j_2)=\mathbf{t}_{n+1}(j_1)$ then $D(\mathbf{t}^{\prime}_{n+1}(j_2))=\mathbf{s}_n(i_2)=D(\mathbf{t}_{n+1}(j_1))=\mathbf{s}_n(i_1))$ so that $i_2=i_1$. It follows that $j_2=j_1+\ell m$ for some integer $\ell$ and $T^{\prime}$ is a translate of $T$. Hence $T^{\prime}=T$ or they are $(n+1)$-window disjoint.

%If $\mathbf{t}^{\prime}_{n+1}(j_2)=\underline{\mathbf{t}}^R_{n+1}(j_1)$
%then
%$D(\mathbf{t}^{\prime}_{n+1}(j_2))=\mathbf{s}_n(i_2)=D(\underline{\mathbf{t}}^R_{n+1}%(j_1))=\mathbf{s}^R_n(i_1))$.
%This is impossible and so $T^{\prime}=T$ or they are mutually negative
%orientable.

If $\mathbf{t}^{\prime}_{n+1}(j_2)=-{\mathbf{t}}^R_{n+1}(j_1)$
then
$D(\mathbf{t}^{\prime}_{n+1}(j_2))=\mathbf{s}_n(i_2)=D(-{\mathbf{t}}^R_{n+1}(j_1))=\mathbf{s}^R_n(i_1)$.
This is impossible and so $T^{\prime}=T$ or they are mutually negative
orientable.

The theorem follows. \qed
\end{proof}

The proof of the following theorem is completely analogous to that of Theorem~\ref{thm Lempel-orientable}.

\begin{theorem}  \label{thm Lempel-negative-orientable}
Suppose $S=(s_i)$ is a negative orientable sequence of order $n$ and period $m$. If $w_q(S)$ has
additive order $h$ as a residue modulo $q$ then $D^{-1}(S)$ consists of $h$ shifts of each of $q/h$
mutually no-disjoint orientable sequences of order $n+1$ and period $hm$ which are translates of
one another.
\end{theorem}

The following result follows immediately from Theorems~\ref{thm Lempel-orientable} and~\ref{thm Lempel-negative-orientable}

\begin{theorem}  \label{thm Lempel-special-orientable}
Suppose $S=(s_i)$ is a special orientable sequence of order $n$ and period $m$. If $w_q(S)$ has
additive order $h$ as a residue modulo $q$ then $D^{-1}(S)$ consists of $h$ shifts of each of $q/h$
mutually special-o-disjoint special orientable sequences of order $n+1$ and period $hm$ which are
translates of one another.
\end{theorem}

\begin{corollary} \label{cor Lempel-orientable}
If $w_q(S)$ has order $q$ then the $q$ shifts in $D^{-1}(S)$ have period $qm$ and each contains a
maximal run $a^{t+1}$ for each $a \in \mathbb{Z}_q$ where $t>0$ is the length of a maximal run of
$0$ in $S$.
\end{corollary}

\begin{proof}
The statement about maximal runs follows because $D(a^{s})=0^{s-1}$ for any $a \in \mathbb{Z}_q$ and $s \ge 2$. \qed
\end{proof}

We now turn our attention to homomorphism $A$.

\begin{theorem}  \label{thm Lempel-A-orientable}
Suppose $S=(s_i)$ is an orientable sequence of order $n$ and period $m$.
If $w_q(S)$ has additive order $h$ as a residue modulo $q$ then
\newline (i) if $m$ is odd and $q$ is odd, $A^{-1}(S)$ consists of one orientable sequence of period $m$ and two shifts of each of $\frac{q-1}{2}$  orientable sequences of period $2m$ which are alternating sign translates of one another.
\newline (ii) if $m$ is odd and $q$ is even and $2x=w^A_q(S)$ for some $x \in \mathbb{Z}_q$, $A^{-1}(S)$ consists of two orientable sequences of period $m$ which are alternating sign translates of each other and two shifts of each of $\frac{q-2}{2}$ orientable sequences of period $2m$ which are alternating sign translates of each other.
\newline (iii) if $m$ is odd and $q$ is even and there is no solution for $x$ to $2x=w^A_q(S)$, $A^{-1}(S)$ consists of two shifts of each of $\frac{q}{2}$ orientable sequences of period $2m$ which are alternating sign translates of each other.
\newline (iv) if $m$ is even and $w^A_q(S)$ has additive order $h$ as a residue modulo $q$, $A^{-1}(S)$ consists of $h$ shifts of $\frac{q}{h}$ mutually o-disjoint orientable sequences of order $n+1$ and of period $hm$ which are  alternating sign translates of one another
\end{theorem}

The proof of this theorem uses arguments similar to those of the proof of Theorem~ \ref{thm Lempel-orientable}. We simply note that if $T=(t_i)$ is an element of $A^{-1}(S)$ then
%\[ t_i=(-1)^i t_0\oplus \bigoplus_{j=0}^{i-1}(-1)^{i-1-j}s_j.  \]
\[ t_i=(-1)^i t_0 + \sum_{j=0}^{i-1}(-1)^{i-1-j}s_j.  \]
It is clear that two elements $T,T^{\prime} \in A^{-1}(S)$ are alternating sign translates of one another. Further if $\mathbf{t}_{n+1}(j_2)=\mathbf{t}^R_{n+1}(j_1)$ then $A(\mathbf{t}_{n+1}(j_2))=\mathbf{s}_{n}(i_2)=A(\mathbf{t}^R_{n+1}(j_1))=\mathbf{s}^R_{n}(i_1)$ where $i_1,i_2 \in \{0,1,\dots,m-1\}$ with $j_1\equiv i_1 \pmod m$ and $j_2\equiv i_2 \pmod m$.

Clearly we achieve the greatest increase in period over that of $S$ for a sequence $T$ in $D^{-1}(S)$, respectively $T$ in $A^{-1}(S)$, when $w_q(S)$, respectively $w^A_q(S)$, has order $q$ in the additive group of residues modulo $q$ (for example $w_q(S)=1$ or $w^A_q(S)=1$). In such a case, since $\{t_i,t_{i+m},\dots,t_{i+(q-1)m}\}=\mathbb{Z}_q$, for each non-negative integer $i$, $w_q(T)$ or $w^A_q(T)$ equals $m\frac{q(q-1)}{2} \pmod q$ which has order $1$ or $2$. Thus $T$ is not a suitable sequence on which to repeat the construction. To achieve maximum scaling of the period  for multiple iterations of the above construction we modify the output sequence $T$ in $D^{-1}(S)$ or $A^{-1}(S)$ to ensure that the resulting modified sequence has suitable weight.

\subsection{An explicit construction}
In this subsection, we present a special case of construction that gives sequences that can be used iteratively to build sequences of higher orders. We will consider a starter sequence $S$ of order $n$ that is special orientable with $w_q(S)\equiv0\mod q$. By Theorem~\ref{thm Lempel-special-orientable}, $D_{\beta}^{-1}(S)$ consists of $q$ mutually special-o-disjoint special orientable sequences of order $n+1$ that are all shifts of each other. The following definitions and results were introduced in Alhakim and Akinwande~\cite{Alhakim11}, where they generalised the Lempel construction to full de~Bruijn sequences of arbitrary alphabet size.

\begin{definition}\label{D:alterString}
For a given nonzero value $\lambda\in\mathbb{Z}_q$ we define the string
$\theta^{(\lambda)}_n =(e_1,\ldots, e_n)$ as $e_1=0$ and
$e_{i+1}=e_i+\lambda$ for $i>1$. A \emph{generalised alternating string} of
size $n$ is the component-wise sum $\theta^{(\lambda)}_n+a^n$, where $a^n$ is a constant string with $a\in\mathbb{Z}_q$. For simplicity, it will be denoted by $\theta^{(\lambda)}_n+a$ and we say that the latter is a translation by $a$ of $\theta^{(\lambda)}_n$.
\end{definition}

By the above definition, the words $\theta^{(\lambda)}_n+i\lambda$
and $\theta^{(\lambda)}_n+(i+1)\lambda$, for an integer $i$, are the
generalised alternating strings $(i\lambda,(i+1)\lambda,\ldots,(i+n-1)\lambda)$
and $((i+1)\lambda,\ldots,(i+n)\lambda)$ respectively, where all
entries are taken modulo $q$. Thus $(\theta^{(\lambda)}_n+i\lambda;
\theta^{(\lambda)}_n+(i+1)\lambda)$ is an edge in $B_n(q)$ for all
values of $i$.

\begin{definition}
For any element $\lambda\neq0$, the alternating cycle associated
with the generalised alternating string $\theta^{(\lambda)}_n$ is induced by the
the edges $(i\lambda+\theta^{(\lambda)}_n;
(i+1)\lambda+\theta^{(\lambda)}_n)$ for $i=0$ to $q-1$.
\end{definition}

It is immediate that an alternating cycle has period $q$ if and only if $\lambda$ is coprime to $q$. Furthermore, Lemma~3.6 in \cite{Alhakim11} states that for any nonzero element $\gamma\in Z_q$, the inverse image of the constant string $\gamma^n$ by $D_{\beta}$ contains precisely the generalized alternating strings that are translates of $\theta^{(\beta^{-1}\gamma)}_{n+1}$.

Returning to the first paragraph above, if $S$ contains a constant string $\gamma^n$, each sequence in $D_{\beta}^{-1}(S)$ will then contain a distinct translate of the generalised alternating string. Furthermore, if $gcd(\gamma,q)=1$, the unique generalized alternating string of each inverse cycle can be written as  $\theta^{(\beta^{-1}\gamma)}_{n+1}+\beta^{-1}\gamma i$. That is, the alternating cycle associated with $\theta^{(\beta^{-1}\gamma)}_{n+1}$ is of period $q$ and has one vertex on each inverse sequence. This is true because $\beta^{-1}\gamma$ is coprime to $q$. Therefore, none of the edges of the alternating cycle is an edge of a cycle in $D_{\beta}^{-1}(S)$. The inverse cycles are joined together by adding all but one edge in the alternating cycle, deleting and adding edges accordingly. We illustrate this construction using $q=5$ and $S=[001112]$, which has zero weight and is orientable of window length $3$ except for the constant string $111$. The inverse cycles by $D_1$ are\\

\noindent $T_0 =  [ 0 \ 0 \ 1 \ 2 \ 3 ]$;
$T_1 = [ 1 \ 1 \ 2 \ 3 \ 4 ]$;
$T_2 = [ 2 \ 2 \ 3 \ 4 \ 0 ]$;
$T_3 = [ 3 \ 3 \ 4 \ 0 \ 1 ]$;
and $T_4 = [ 4 \ 4 \ 0 \ 1 \ 2 ]$.\newline

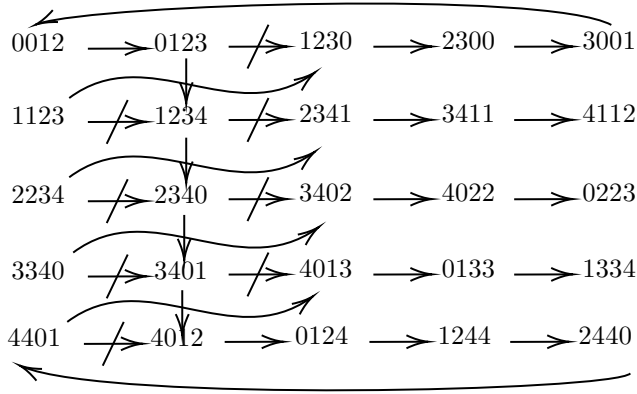
\tikzset{every picture/.style={line width=0.75pt}} %set default line width to 0.75pt

\begin{tikzpicture}[x=0.75pt,y=0.75pt,yscale=-1,xscale=1]
%uncomment if require: \path (0,380); %set diagram left start at 0, and has height of 380

%Straight Lines [id:da2596547340224351]
\draw    (92,40) -- (119.6,39.91) ;
\draw [shift={(121.6,39.9)}, rotate = 179.81] [color={rgb, 255:red, 0; green, 0; blue, 0 }  ][line width=0.75]    (10.93,-3.29) .. controls (6.95,-1.4) and (3.31,-0.3) .. (0,0) .. controls (3.31,0.3) and (6.95,1.4) .. (10.93,3.29)   ;
%Straight Lines [id:da5667674988672142]
\draw    (307,39) -- (334.6,38.91) ;
\draw [shift={(336.6,38.9)}, rotate = 179.81] [color={rgb, 255:red, 0; green, 0; blue, 0 }  ][line width=0.75]    (10.93,-3.29) .. controls (6.95,-1.4) and (3.31,-0.3) .. (0,0) .. controls (3.31,0.3) and (6.95,1.4) .. (10.93,3.29)   ;
%Straight Lines [id:da9850647880701506]
\draw    (236,39) -- (263.6,38.91) ;
\draw [shift={(265.6,38.9)}, rotate = 179.81] [color={rgb, 255:red, 0; green, 0; blue, 0 }  ][line width=0.75]    (10.93,-3.29) .. controls (6.95,-1.4) and (3.31,-0.3) .. (0,0) .. controls (3.31,0.3) and (6.95,1.4) .. (10.93,3.29)   ;
%Straight Lines [id:da13018256962281582]
\draw    (163,39) -- (190.6,38.91) ;
\draw [shift={(192.6,38.9)}, rotate = 179.81] [color={rgb, 255:red, 0; green, 0; blue, 0 }  ][line width=0.75]    (10.93,-3.29) .. controls (6.95,-1.4) and (3.31,-0.3) .. (0,0) .. controls (3.31,0.3) and (6.95,1.4) .. (10.93,3.29)   ;
%Curve Lines [id:da6300886038936893]
\draw    (356,28) .. controls (322.29,10.26) and (97.84,17.59) .. (67.23,27.3) ;
\draw [shift={(65.6,27.9)}, rotate = 336.5] [color={rgb, 255:red, 0; green, 0; blue, 0 }  ][line width=0.75]    (10.93,-3.29) .. controls (6.95,-1.4) and (3.31,-0.3) .. (0,0) .. controls (3.31,0.3) and (6.95,1.4) .. (10.93,3.29)   ;
%Straight Lines [id:da7852645999909009]
\draw    (92,77) -- (119.6,76.91) ;
\draw [shift={(121.6,76.9)}, rotate = 179.81] [color={rgb, 255:red, 0; green, 0; blue, 0 }  ][line width=0.75]    (10.93,-3.29) .. controls (6.95,-1.4) and (3.31,-0.3) .. (0,0) .. controls (3.31,0.3) and (6.95,1.4) .. (10.93,3.29)   ;
%Straight Lines [id:da010971648385235566]
\draw    (307,76) -- (334.6,75.91) ;
\draw [shift={(336.6,75.9)}, rotate = 179.81] [color={rgb, 255:red, 0; green, 0; blue, 0 }  ][line width=0.75]    (10.93,-3.29) .. controls (6.95,-1.4) and (3.31,-0.3) .. (0,0) .. controls (3.31,0.3) and (6.95,1.4) .. (10.93,3.29)   ;
%Straight Lines [id:da5812250818187232]
\draw    (236,76) -- (263.6,75.91) ;
\draw [shift={(265.6,75.9)}, rotate = 179.81] [color={rgb, 255:red, 0; green, 0; blue, 0 }  ][line width=0.75]    (10.93,-3.29) .. controls (6.95,-1.4) and (3.31,-0.3) .. (0,0) .. controls (3.31,0.3) and (6.95,1.4) .. (10.93,3.29)   ;
%Straight Lines [id:da667364983288325]
\draw    (163,76) -- (190.6,75.91) ;
\draw [shift={(192.6,75.9)}, rotate = 179.81] [color={rgb, 255:red, 0; green, 0; blue, 0 }  ][line width=0.75]    (10.93,-3.29) .. controls (6.95,-1.4) and (3.31,-0.3) .. (0,0) .. controls (3.31,0.3) and (6.95,1.4) .. (10.93,3.29)   ;
%Straight Lines [id:da8532894435598462]
\draw    (92,117) -- (119.6,116.91) ;
\draw [shift={(121.6,116.9)}, rotate = 179.81] [color={rgb, 255:red, 0; green, 0; blue, 0 }  ][line width=0.75]    (10.93,-3.29) .. controls (6.95,-1.4) and (3.31,-0.3) .. (0,0) .. controls (3.31,0.3) and (6.95,1.4) .. (10.93,3.29)   ;
%Straight Lines [id:da7411845093399256]
\draw    (307,116) -- (334.6,115.91) ;
\draw [shift={(336.6,115.9)}, rotate = 179.81] [color={rgb, 255:red, 0; green, 0; blue, 0 }  ][line width=0.75]    (10.93,-3.29) .. controls (6.95,-1.4) and (3.31,-0.3) .. (0,0) .. controls (3.31,0.3) and (6.95,1.4) .. (10.93,3.29)   ;
%Straight Lines [id:da2116854572639022]
\draw    (236,116) -- (263.6,115.91) ;
\draw [shift={(265.6,115.9)}, rotate = 179.81] [color={rgb, 255:red, 0; green, 0; blue, 0 }  ][line width=0.75]    (10.93,-3.29) .. controls (6.95,-1.4) and (3.31,-0.3) .. (0,0) .. controls (3.31,0.3) and (6.95,1.4) .. (10.93,3.29)   ;
%Straight Lines [id:da5062235640814086]
\draw    (163,116) -- (190.6,115.91) ;
\draw [shift={(192.6,115.9)}, rotate = 179.81] [color={rgb, 255:red, 0; green, 0; blue, 0 }  ][line width=0.75]    (10.93,-3.29) .. controls (6.95,-1.4) and (3.31,-0.3) .. (0,0) .. controls (3.31,0.3) and (6.95,1.4) .. (10.93,3.29)   ;
%Straight Lines [id:da8335987005443186]
\draw    (92,155) -- (119.6,154.91) ;
\draw [shift={(121.6,154.9)}, rotate = 179.81] [color={rgb, 255:red, 0; green, 0; blue, 0 }  ][line width=0.75]    (10.93,-3.29) .. controls (6.95,-1.4) and (3.31,-0.3) .. (0,0) .. controls (3.31,0.3) and (6.95,1.4) .. (10.93,3.29)   ;
%Straight Lines [id:da008820638828584748]
\draw    (307,154) -- (334.6,153.91) ;
\draw [shift={(336.6,153.9)}, rotate = 179.81] [color={rgb, 255:red, 0; green, 0; blue, 0 }  ][line width=0.75]    (10.93,-3.29) .. controls (6.95,-1.4) and (3.31,-0.3) .. (0,0) .. controls (3.31,0.3) and (6.95,1.4) .. (10.93,3.29)   ;
%Straight Lines [id:da9816861359487026]
\draw    (236,154) -- (263.6,153.91) ;
\draw [shift={(265.6,153.9)}, rotate = 179.81] [color={rgb, 255:red, 0; green, 0; blue, 0 }  ][line width=0.75]    (10.93,-3.29) .. controls (6.95,-1.4) and (3.31,-0.3) .. (0,0) .. controls (3.31,0.3) and (6.95,1.4) .. (10.93,3.29)   ;
%Straight Lines [id:da3794273846715588]
\draw    (163,154) -- (190.6,153.91) ;
\draw [shift={(192.6,153.9)}, rotate = 179.81] [color={rgb, 255:red, 0; green, 0; blue, 0 }  ][line width=0.75]    (10.93,-3.29) .. controls (6.95,-1.4) and (3.31,-0.3) .. (0,0) .. controls (3.31,0.3) and (6.95,1.4) .. (10.93,3.29)   ;
%Straight Lines [id:da8121806845395942]
\draw    (90,189) -- (117.6,188.91) ;
\draw [shift={(119.6,188.9)}, rotate = 179.81] [color={rgb, 255:red, 0; green, 0; blue, 0 }  ][line width=0.75]    (10.93,-3.29) .. controls (6.95,-1.4) and (3.31,-0.3) .. (0,0) .. controls (3.31,0.3) and (6.95,1.4) .. (10.93,3.29)   ;
%Straight Lines [id:da6536452199574563]
\draw    (305,188) -- (332.6,187.91) ;
\draw [shift={(334.6,187.9)}, rotate = 179.81] [color={rgb, 255:red, 0; green, 0; blue, 0 }  ][line width=0.75]    (10.93,-3.29) .. controls (6.95,-1.4) and (3.31,-0.3) .. (0,0) .. controls (3.31,0.3) and (6.95,1.4) .. (10.93,3.29)   ;
%Straight Lines [id:da7577315165051239]
\draw    (234,188) -- (261.6,187.91) ;
\draw [shift={(263.6,187.9)}, rotate = 179.81] [color={rgb, 255:red, 0; green, 0; blue, 0 }  ][line width=0.75]    (10.93,-3.29) .. controls (6.95,-1.4) and (3.31,-0.3) .. (0,0) .. controls (3.31,0.3) and (6.95,1.4) .. (10.93,3.29)   ;
%Straight Lines [id:da30523266296304574]
\draw    (161,188) -- (188.6,187.91) ;
\draw [shift={(190.6,187.9)}, rotate = 179.81] [color={rgb, 255:red, 0; green, 0; blue, 0 }  ][line width=0.75]    (10.93,-3.29) .. controls (6.95,-1.4) and (3.31,-0.3) .. (0,0) .. controls (3.31,0.3) and (6.95,1.4) .. (10.93,3.29)   ;
%Curve Lines [id:da912440939265323]
\draw    (365.6,203.9) .. controls (356.69,212.81) and (102.76,218.78) .. (58.87,200.46) ;
\draw [shift={(57.6,199.9)}, rotate = 25.41] [color={rgb, 255:red, 0; green, 0; blue, 0 }  ][line width=0.75]    (10.93,-3.29) .. controls (6.95,-1.4) and (3.31,-0.3) .. (0,0) .. controls (3.31,0.3) and (6.95,1.4) .. (10.93,3.29)   ;
%Straight Lines [id:da7632991130709337]
\draw    (141.6,44.9) -- (141.6,65.9) ;
\draw [shift={(141.6,67.9)}, rotate = 270] [color={rgb, 255:red, 0; green, 0; blue, 0 }  ][line width=0.75]    (10.93,-3.29) .. controls (6.95,-1.4) and (3.31,-0.3) .. (0,0) .. controls (3.31,0.3) and (6.95,1.4) .. (10.93,3.29)   ;
%Straight Lines [id:da5704937184275918]
\draw    (139.6,161.9) -- (139.6,182.9) ;
\draw [shift={(139.6,184.9)}, rotate = 270] [color={rgb, 255:red, 0; green, 0; blue, 0 }  ][line width=0.75]    (10.93,-3.29) .. controls (6.95,-1.4) and (3.31,-0.3) .. (0,0) .. controls (3.31,0.3) and (6.95,1.4) .. (10.93,3.29)   ;
%Straight Lines [id:da39609413041569885]
\draw    (140.6,123.9) -- (140.6,144.9) ;
\draw [shift={(140.6,146.9)}, rotate = 270] [color={rgb, 255:red, 0; green, 0; blue, 0 }  ][line width=0.75]    (10.93,-3.29) .. controls (6.95,-1.4) and (3.31,-0.3) .. (0,0) .. controls (3.31,0.3) and (6.95,1.4) .. (10.93,3.29)   ;
%Straight Lines [id:da42978785143581266]
\draw    (141.6,84.9) -- (141.6,105.9) ;
\draw [shift={(141.6,107.9)}, rotate = 270] [color={rgb, 255:red, 0; green, 0; blue, 0 }  ][line width=0.75]    (10.93,-3.29) .. controls (6.95,-1.4) and (3.31,-0.3) .. (0,0) .. controls (3.31,0.3) and (6.95,1.4) .. (10.93,3.29)   ;
%Straight Lines [id:da9046117870998778]
\draw    (101.5,88) -- (112.1,65.9) ;
%Straight Lines [id:da042351004776038925]
\draw    (172.5,165) -- (183.1,142.9) ;
%Straight Lines [id:da5001867323249394]
\draw    (172.5,127) -- (183.1,104.9) ;
%Straight Lines [id:da3895785699155869]
\draw    (172.5,87) -- (183.1,64.9) ;
%Straight Lines [id:da9992811399565353]
\draw    (172.5,50) -- (183.1,27.9) ;
%Straight Lines [id:da8881155422361815]
\draw    (99.5,200) -- (110.1,177.9) ;
%Straight Lines [id:da8807814983482543]
\draw    (101.5,166) -- (112.1,143.9) ;
%Straight Lines [id:da4778550174714926]
\draw    (101.5,128) -- (112.1,105.9) ;
%Curve Lines [id:da1723190261131402]
\draw    (81.6,177.9) .. controls (121.2,148.2) and (164.72,193.97) .. (204.4,165.78) ;
\draw [shift={(205.6,164.9)}, rotate = 143.13] [color={rgb, 255:red, 0; green, 0; blue, 0 }  ][line width=0.75]    (10.93,-3.29) .. controls (6.95,-1.4) and (3.31,-0.3) .. (0,0) .. controls (3.31,0.3) and (6.95,1.4) .. (10.93,3.29)   ;
%Curve Lines [id:da7476222388869529]
\draw    (82.6,64.9) .. controls (122.2,35.2) and (165.72,80.97) .. (205.4,52.78) ;
\draw [shift={(206.6,51.9)}, rotate = 143.13] [color={rgb, 255:red, 0; green, 0; blue, 0 }  ][line width=0.75]    (10.93,-3.29) .. controls (6.95,-1.4) and (3.31,-0.3) .. (0,0) .. controls (3.31,0.3) and (6.95,1.4) .. (10.93,3.29)   ;
%Curve Lines [id:da9592304891985348]
\draw    (82.6,104.9) .. controls (122.2,75.2) and (165.72,120.97) .. (205.4,92.78) ;
\draw [shift={(206.6,91.9)}, rotate = 143.13] [color={rgb, 255:red, 0; green, 0; blue, 0 }  ][line width=0.75]    (10.93,-3.29) .. controls (6.95,-1.4) and (3.31,-0.3) .. (0,0) .. controls (3.31,0.3) and (6.95,1.4) .. (10.93,3.29)   ;
%Curve Lines [id:da35929548477451556]
\draw    (84.6,142.9) .. controls (124.2,113.2) and (167.72,158.97) .. (207.4,130.78) ;
\draw [shift={(208.6,129.9)}, rotate = 143.13] [color={rgb, 255:red, 0; green, 0; blue, 0 }  ][line width=0.75]    (10.93,-3.29) .. controls (6.95,-1.4) and (3.31,-0.3) .. (0,0) .. controls (3.31,0.3) and (6.95,1.4) .. (10.93,3.29)   ;

% Text Node
\draw (52,31) node [anchor=north west][inner sep=0.75pt]   [align=left] {0012};
% Text Node
\draw (124,31) node [anchor=north west][inner sep=0.75pt]   [align=left] {0123};
% Text Node
\draw (340,30) node [anchor=north west][inner sep=0.75pt]   [align=left] {3001};
% Text Node
\draw (269,30) node [anchor=north west][inner sep=0.75pt]   [align=left] {2300};
% Text Node
\draw (197,30) node [anchor=north west][inner sep=0.75pt]   [align=left] {1230};
% Text Node
\draw (52,68) node [anchor=north west][inner sep=0.75pt]   [align=left] {1123};
% Text Node
\draw (124,68) node [anchor=north west][inner sep=0.75pt]   [align=left] {1234};
% Text Node
\draw (340,67) node [anchor=north west][inner sep=0.75pt]   [align=left] {4112};
% Text Node
\draw (269,67) node [anchor=north west][inner sep=0.75pt]   [align=left] {3411};
% Text Node
\draw (197,67) node [anchor=north west][inner sep=0.75pt]   [align=left] {2341};
% Text Node
\draw (52,108) node [anchor=north west][inner sep=0.75pt]   [align=left] {2234};
% Text Node
\draw (124,108) node [anchor=north west][inner sep=0.75pt]   [align=left] {2340};
% Text Node
\draw (340,107) node [anchor=north west][inner sep=0.75pt]   [align=left] {0223};
% Text Node
\draw (269,107) node [anchor=north west][inner sep=0.75pt]   [align=left] {4022};
% Text Node
\draw (197,107) node [anchor=north west][inner sep=0.75pt]   [align=left] {3402};
% Text Node
\draw (52,146) node [anchor=north west][inner sep=0.75pt]   [align=left] {3340};
% Text Node
\draw (124,146) node [anchor=north west][inner sep=0.75pt]   [align=left] {3401};
% Text Node
\draw (340,145) node [anchor=north west][inner sep=0.75pt]   [align=left] {1334};
% Text Node
\draw (269,145) node [anchor=north west][inner sep=0.75pt]   [align=left] {0133};
% Text Node
\draw (197,145) node [anchor=north west][inner sep=0.75pt]   [align=left] {4013};
% Text Node
\draw (50,180) node [anchor=north west][inner sep=0.75pt]   [align=left] {4401};
% Text Node
\draw (122,180) node [anchor=north west][inner sep=0.75pt]   [align=left] {4012};
% Text Node
\draw (338,179) node [anchor=north west][inner sep=0.75pt]   [align=left] {2440};
% Text Node
\draw (267,179) node [anchor=north west][inner sep=0.75pt]   [align=left] {1244};
% Text Node
\draw (195,179) node [anchor=north west][inner sep=0.75pt]   [align=left] {0124};
\end{tikzpicture}

%\end{tabular}\newline

These $4$-window cycles are $o$-disjoint and can therefore be `stitched together' into one
orientable sequence of order $4$. We start along the top sequence, then move along the alternating sequence corresponding to $0123$, crossing each cycle exactly at one vertex. Reaching down to the last cycle, we exhaust that cycle and work our way up to the previous cycle and so on. This is illustrated in the diagram on the right. The resulting cycle is (where a substring between dots comes from one inverse cycle):
\[
00123.4.0.1.2.44401.33340.22234.11123.0
\]
The resulting sequence is orientable, but not negative orientable. To use it for another iteration, it just needs to be pruned around the joining points to recover the negative orientable property, keeping the zero weight of the sequence and making sure that, e.g.\ 1111 is a string. One such valid alteration is
\[
0023.4.1.2.44401.33340.22234.1111.0
\]

Observe that the use of cycle-joining --- in particular to construct de Bruijn sequences --- has a long history; for further details see, for example, Sawada et al. \cite{Sawada23}.

%%%%%%%%%

\subsection{Maintaining maximum order}

The modification of the sequence $T$ arising from the inverse application of the homomorphism $D$
or $A$ that we use introduces one or two extra terms into its ring sequence to obtain a new
ring sequence. By this means we obtain a sequence whose weight differs from that of $T$. By
convention, in the following definition and elsewhere, should $j=0$ we understand $s_{-1}$ to be
$s_{m-1}$, where $m$ is the period of the sequence $(s_i)$.

We begin by considering the application of the inverse of $D$ recursively.

\begin{definition}
Let $S=(s_i)$ be a periodic $q$-ary sequence. For $a \in \mathbb{Z}_q$ we write $a^t$ for a string of $t$ consecutive terms $a$ of $S$. A {\emph run} of $a$ in $S$ is a string  $s_j,\dots,s_{j+t-1}=a^t$ with $s_{j-1},s_{j+t} \not=a$. A run, $a^t$, is {\emph maximal} in $S$ if any string $a^{t^{\prime}}$ of $S$ has $t^{\prime} \le t$  and further is {\emph inverse maximal} if, in addition, any string $1^{t^{\prime}} + (q-1-a)^{t^{\prime}}$ of $S$ has $t^{\prime} \le t$.
\end{definition}

We note that if $a^t$ is a maximal run of $S$ then $(a + b)^t$ is a maximal run of the translate by $b$ of $S$ and there exists a shift of such a translate that has ring sequence whose first $t$ terms are $a + b$. Moreover if $S$ is orientable (respectively negative orientable) then so is this shifted translate.

\begin{definition}
Let $ a \in \mathbb{Z}_q$. Suppose that the ring sequence of a periodic sequence $S$ is $[s_0,s_1,\cdots,s_{m-1}]$ and $r$ is the smallest non-negative integer such that
$a^t=s_r,s_{r+1},\cdots,s_{r+t-1}$ is a maximal run for $a \in \mathbb{Z}_q$. Define the sequence $\mathcal{E}_a(S)$   to be the sequence with
ring sequence
 \[ [s_0,s_1,\cdots,s_{r-1},a,s_r,s_{r+1},\cdots,s_{m-1}] \]
i.e.\ where the occurrence of $a^t$ is replaced with $a^{t+1}$.
\end{definition}

\begin{definition}
Let $ a \in \mathbb{Z}_q$. Suppose that the ring sequence of a periodic sequence $S$ is $[s_0,s_1,\ldots,s_{m-1}]$ and $r$ is the smallest non-negative integer such that
$a^t=s_r,s_{r+1},\cdots,s_{r+t-1}$ is a maximal run for $a= 1-w_q(S)$. If $a=0$ define the sequence $\mathcal{E}(S)$  to be $S$, and if $a \not=0$ define $\mathcal{E}(S)$ to be $\mathcal{E}_a(S)$.
\end{definition}

We can now state a key result.

\begin{lemma}  \label{lemma-E}
Suppose $S=(s_i)$ is an orientable (respectively negative orientable) sequence of order $n$ and
period $m$. If $w_q(S) \not=1$ and $S$ does not contain a string $a^{n-1}$ for
$a=1-w_q(S)$ then $\mathcal{E}_a(S)$ is an orientable (respectively negative  orientable)
sequence of order $n$, period  $m+1$, and having $w_q(\mathcal{E}_a(S))=1$.
\end{lemma}

\begin{proof}
Suppose that $w_q(S) \not=1$ so that $a=1-w_q(S)\not=0$.
Let $t$ be the length of a maximal run of $a$ in $S$. The fact
that $w_q(\mathcal{E}_a(S))=1$ follows immediately from the above definitions. Again by definition
the period of $\mathcal{E}_a(S)$ divides $m+1$, and is precisely $m+1$ because its ring sequence
$[s_0,s_1,\cdots,s_{r-1},a,s_r,s_{r+1},\cdots,s_{m-1}]$ contains exactly one occurrence of $a^{t+1}$ since $t$ is the length of a maximal run of $a$ in $S$ and only one occurrence of $a^t$ is replaced by $a^{t+1}$ to obtain $\mathcal{E}_a(S)$.

It remains to show that $\mathcal{E}_a(S)$ is an orientable (respectively  negative orientable) sequence.  We only need to examine the
$n$-tuples which include the string containing the inserted element.  Inserting this single element means that the following
$n-1-t$ $n$-tuples that occur in $S$ (where the subscripts are computed modulo $m$):
\[ \mathbf{u}_i  =  (s_{r-i},\dots, s_{r-1}, a^t, s_{r+t}, \dots,s_{r+n-i-1}),~~~(1\leq i\leq n-1-t), \]
are replaced in $\mathcal{E}_a(S)$ by the following $n-t$ $n$-tuples:
\[ \mathbf{v}_i =  (s_{r-i},\dots, s_{r-1}, a^{t+1},s_{r+t},\dots, s_{r+n-i-2}),~~~(0\leq i\leq n-1-t). \]

Now all  of the $\mathbf{v}_i$ tuples contain a run $a^{t+1}$, and hence they are all distinct and cannot occur in $S$ and nor do their reverses (respectively negative reverses). The only remaining task is to show that no $\mathbf{v}_i$ equals  $\mathbf{v}_j^R$ (respectively $-\mathbf{{v}}_j^R$) for any $i$ and $j$. As each $n$-tuple contains $a^{t+1}$ and $t+1 \le n-1$ this follows if $\mathbf{v}_i \not= \mathbf{v}_{n-1-t-i}^R$ (respectively $\mathbf{v}_i\not=-\mathbf{v}_{n-1-t-i}^R$), $i=0,1,\dots,\lfloor \frac{n-1-t}{2} \rfloor$. However, if $\mathbf{v}_i=\mathbf{v}_{n-1-t-i}^R$ (respectively $\mathbf{v}_i= -\mathbf{v}_{n-1-t-i}^R$), for some $i$ then it immediately follows that $\mathbf{u}_{\lfloor \frac{n-1-t}{2} \rfloor}=\mathbf{u}^R_{\lfloor \frac{n-t}{2} \rfloor}$ (respectively $\mathbf{u}_{\lfloor \frac{n-1-t}{2} \rfloor}=-\mathbf{u}^R_{\lfloor \frac{n-t}{2} \rfloor}$), which contradicts the assumption on $S$. \qed
\end{proof}

This then enables us to give a means to recursively generate `long' orientable sequences.  We first
define a subclass of orientable (respectively negative orientable) sequences

\begin{definition}
An orientable (respectively negative orientable) sequence with the property that any run of $0$ has
length at most $n-2$ is said to be \emph{good}.
\end{definition}

\begin{theorem}
Suppose $S=(s_i)$ is a good orientable (respectively negative orientable) sequence of order $n$, period $m$, and with $w_q(S)=1$. If $T\in
D^{-1}(S)$ then, for $a =1-w_q(T)$, $\mathcal{E}_a(T)$ is a good negative orientable (respectively  orientable) sequence of order $n+1$,
period $qm+1$, and with $w_q(\mathcal{E}_a(T))=1$.
\end{theorem}

\begin{proof}
The fact that $T$ is a negative orientable (respectively orientable) sequence of order $n+1$ and period $qm$ follows immediately from
Theorem~\ref{thm Lempel-orientable} (respectively Theorem~\ref{thm Lempel-negative-orientable}) and Corollary~\ref{cor Lempel-orientable}; we also know the weight of $T$ is $m\frac{q(q-1)}{2}\not=1$.  Since $0^t$ occurs in $S$, each of $a^{t+1}$ for $a \in \mathbb{Z}_q$ occurs in $T$, as
$D^{-1}(0^{t})=\{a^{t+1} | a \in \mathbb{Z_q}\}$. Moreover, for each $a \in \mathbb{Z}_q$, $a^{t+1}$ is a maximal run in $T$  as $D(a^{t^{\prime}+1})=0^{t^{\prime}}$ for any $a^{t^{\prime}+1}$ in $T$. This means that $T$ is good and also the conditions of
Lemma~\ref{lemma-E} apply.  This in turn means that $\mathcal{E}_a(T)$ is a negative orientable (respectively  orientable) sequence of
order $n+1$ and weight $1$.  The fact that $\mathcal{E}_a(T)$ is good follows from observing that
applying $\mathcal{E}_a$ cannot affect the occurrences of runs of $0$.  Finally,
$\mathcal{E}_a(T)$ has period $km+1$. \qed
\end{proof}

This immediately gives the following result.

\begin{corollary}  \label{Cor-length1}
Suppose $S_n$ is a good orientable (respectively negative orientable) sequence  of order $n$ and period $m_n$ and with $w_q(S)=1$. Recursively define the
sequences $S_{i+1}=\mathcal{E}_a(D^{-1}(S_i))$, where $a = 1-w_q(D^{-1}(S_i))$, for $i\geq n$, and suppose $S_i$ has period $m_i$
($i>n$). Then, $S_{i}$ is a good orientable or negative orientable sequence for every $i\ge n$, and
$m_{n+j+1}=qm_{n+j}+1$ for every $j\geq 0$. $S_{i}$ is an orientable (respectively negative orientable) sequence when $i-n$ is even and a negative orientable (respectively orientable) sequence when $i-n$ is odd.
\end{corollary}

\begin{proof}
Now, for $i \ge n$, $\mathcal{E}_a(D^{-1}(S_{i}))$ has weight equal to $1$ and so $S_{i+1}$ has period $qm_i+1$. This immediately yields the result. \qed
\end{proof}

Simple numerical calculations give the following.

\begin{corollary}
Suppose the sequences $(S_i)$ are defined as in Corollary~\ref{Cor-length1}.  Then
$m_{n+j}=q^{j}m_n+\frac{q^{j}-1}{q-1}$.
\end{corollary}

We now consider sequences generated by applying the inverse of homomorphism $A$ recursively.

\begin{definition}
Let $S=(s_i)$ be a periodic $q$-ary sequence. For $a \in \mathbb{Z}_q$ we write $\underline{a}^t$ for a string of $t$ consecutive terms alternating $a$ and $-a$, beginning with $a$, of $S$. A {\emph signed run} of $a$ in $S$ is a string  $s_j,\dots,s_{j+t-1}=\underline{a}^t$ with $s_{j-1}\not=-a$ and $s_{j+t} \not=a$ if $t$ is even and $s_{j+t} \not=-a$ if $t$ is odd. A signed run, $\underline{a}^t$, is said to be \emph{maximal} in $S$ if any string $\underline{a}^{t^{\prime}}$ of $S$ has $t^{\prime} \le t$, and further is said to be \emph{inverse maximal} if also any string $\underline{-a}^{t^{\prime}}$ of $S$ has $t^{\prime} \le t$.
\end{definition}

We note that if $\underline{a}^t$ is a maximal run of $S$ then $\underline{(a + b)}^t$ is a maximal run of the signed translate by $b$ of $S$ and there exists a shift of such a translate that has ring sequence whose first $t$ terms form the string $\underline{a + b}^t$. Moreover if $S$ is orientable  then so is this shifted translate.

%\begin{definition}
%Suppose that the generating cycle of a periodic sequence $S$ is $[s_0,s_1,\ldots,s_{m-1}]$ and $r$ is the %smallest non-negative integer such that
%$\underline{a}^t=s_r,s_{r+1}=\cdots,s_{r+t-1}$ is a maximal signed run for $a=2\oplus (q-1-w_q(S))$. If %$a=0$ define the sequence $\mathcal{E}(S)$  to be $S$ and if $a \not=0$ define $\mathcal{E}(S)$ to be the %sequence with generating cycle
% \[ [s_0,s_1,s_{r-1},a,k-a,s_r,s_{r+1},...s_{m-1}] \]
%i.e.\ where the occurrence of $\underline{a}^t$ is replaced with $\underline{a}^{t+2}$.
%\end{definition}

\begin{definition}
Let $ a \in \mathbb{Z}_q$. Suppose that the ring sequence of a periodic sequence $S$ is $[s_0,s_1,\ldots,s_{m-1}]$ and $r$ is the smallest non-negative integer such that
$\underline{a}^t=s_r,s_{r+1},\cdots,s_{r+t-1}$ is a maximal signed run for $a \in \mathbb{Z}_q$. Define the sequence $\mathcal{E}_a(S)$   to be the sequence with
ring sequence
% \[ [s_0,s_1,s_{r-1},a,k-a,s_r,s_{r+1},...s_{m-1}] \]
\[ [s_0,s_1,\cdots,s_{r-1},a,-a,s_r,s_{r+1},\cdots,s_{m-1}] \]
i.e.\ where the occurrence of $\underline{a}^t$ is replaced with $\underline{a}^{t+2}$.
\end{definition}

We have the following result.

\begin{lemma}  \label{lemma-EX}
Suppose $S=(s_i)$ is an orientable sequence of order $n$ and period $m$. If $w_q(S) \not=1$ and
there exists $a \in \mathbb{Z}_q$ with $w_q(S)-1+2a=0$ and $S$ does not contain a string $a^{n-1}$
then $\mathcal{E}_a(S)$ is an orientable sequence of order $n$, period  $m+2$, and having
$w_q(\mathcal{E}_a(S))=1$ or $w_q(\mathcal{E}_a(S))=q-1$ .
\end{lemma}

\begin{proof}
 Suppose that $w_q(S) \not=1$. If $q$ is odd or $q$ is even and $w_q(S)$ is odd then $a$ with $w_q(S)-1+2a=0$ exists and $a \not=0.$ Let $t$ be the length of a maximal signed run of $a$ in $S$. The fact
that $w_q(\mathcal{E}_a(S))=1$ or $-1$ follows because the two term $a$ and $q-a$ introduced by application of $\mathcal{E}$ are added with opposite signs in $w_q(\mathcal{E}_a(S))$. By definition
the period of $\mathcal{E}_a(S)$ divides $m+2$, and is precisely $m+2$ because its ring sequence
%$[s_0,s_1,s_{r-1},a,q-a,s_r,s_{r+1},...,s_{m-1}]$
\[ [s_0,s_1,s_{r-1},a,-a,s_r,s_{r+1},\cdots,s_{m-1}] \]
contains exactly one occurrence of $\underline{a}^{t+1}$.

It remains to show that $\mathcal{E}_a(S)$ is an orientable sequence.  We only need to examine the
$n$-tuples that include the string containing the inserted elements.  Inserting these two elements means that the following
$n-1-t$ $n$-tuples that occur in $S$ (where the subscripts are computed modulo $m$):
\[ \mathbf{u}_i  =  (s_{r-i},\dots, s_{r-1}, \underline{a}^t, s_{r+t}, \dots,s_{r+n-i-1}),~~~(1\leq i\leq n-1-t), \]
are replaced in $\mathcal{E}_a(S)$ by the following $n+1-t$ $n$-tuples:
\begin{eqnarray*}
\mathbf{v}_{-1} & = & (\underline{(-a)}^{t+1},s_{r+t}, \dots,s_{r+n-2})\\
\mathbf{v}_i & = & (s_{r-i},\dots, s_{r-1}, \underline{a}^{t+2},s_{r+t},\dots, s_{r+n-i-3}),~~~(0\leq i\leq n-2-t),\\
\mathbf{v}_{n-1-t} & = & (s_{r-n+t+1},\dots,s_{r-1},\underline{a}^{t+1}).
\end{eqnarray*}

Now the $\mathbf{v}_i$ tuples that contain a run $a^{t+2}$ are all distinct and cannot
occur in $S$ and nor do their reverses. Moreover, since $\underline{a}^t$ is maximal, $\mathbf{v}_{-1}$ and $\mathbf{v}_{n-1-t}$ are distinct and  are distinct from the other $\mathbf{v}_i$ and cannot occur in $S$. The only remaining task is to show that no $\mathbf{v}_i$ equals  $\mathbf{v}_j^R$ for any $i$ and $j$. This follows if  $\mathbf{v}_{-1} \not= \mathbf{v}_{n-1-t}^R$ and $\mathbf{v}_i \not= \mathbf{v}_{n-2-t-i}^R$, $i=0,\dots,\lfloor \frac{n-2-t}{2} \rfloor$. However, if such an equality occurred then it follows that  $\mathbf{v}_{\lfloor \frac{n-2-t}{2} \rfloor}=\mathbf{v}^R_{\lfloor \frac{n-1-t}{2} \rfloor}$, which contradicts $S$ being orientable. \qed
\end{proof}

\section{Conclusions and future work}

We have developed recursive methods for generating an orientable sequence of order $n+1$ from one
of order $n$, as long as the order $n$ sequence satisfies certain key properties. To date, we do
not have general methods of generating the required `starter' sequences to enable use of these
recursive constructions. Developing such sequences remains a topic for further study.

%\bibliographystyle{plain}
%\bibliography{Coding}

\begin{thebibliography}{10}

\bibitem{Alhakim11} A.~Alhakim and M.~Akinwande.
\newblock A recursive construction of nonbinary de {Bruijn} sequences.
\newblock {\em Des. Codes Cryptogr.}, 60(2):155--169, 2011.

\bibitem{Berkowitz16} R.~Berkowitz and S.~Kopparty.
\newblock Robust positioning patterns.
\newblock In R.~Krauthgamer, editor, {\em Proceedings of the Twenty-Seventh
  Annual {ACM-SIAM} Symposium on Discrete Algorithms, {SODA} 2016, Arlington,
  VA, USA, January 10--12, 2016}, pages 1937--1951. {SIAM}, 2016.

\bibitem{Bruckstein12} A.~M. Bruckstein, T.~Etzion, R.~Giryes, N.~Gordon, R.~J.
    Holt, and
  D.~Shuldiner.
\newblock Simple and robust binary self-location patterns.
\newblock {\em {IEEE} Trans. Inf. Theory}, 58(7):4884--4889, 2012.

\bibitem{Burns92} J.~Burns and C.~J. Mitchell.
\newblock Coding schemes for two-dimensional position sensing.
\newblock Technical Report HPL--92--19, January 1992.
\newblock \url{https://www.chrismitchell.net/HPL-92-19.pdf}.

\bibitem{Burns93} J.~Burns and C.~J. Mitchell.
\newblock Coding schemes for two-dimensional position sensing.
\newblock In M.~J. Ganley, editor, {\em Cryptography and Coding III}, pages
  31--66. Oxford University Press, 1993.

\bibitem{Chee19} Y.~M. Chee, D.~T. Dao, H.~M. Kiah, S.~Ling, and H.~Wei.
\newblock Binary robust positioning patterns with low redundancy and efficient
  locating algorithms.
\newblock In T.~M. Chan, editor, {\em Proceedings of the Thirtieth Annual
  {ACM-SIAM} Symposium on Discrete Algorithms, {SODA} 2019, San Diego,
  California, USA, January 6--9, 2019}, pages 2171--2184. {SIAM}, 2019.

\bibitem{Chee20} Y.~M. Chee, D.~T. Dao, H.~M. Kiah, S.~Ling, and H.~Wei.
\newblock Robust positioning patterns with low redundancy.
\newblock {\em {SIAM} J. Comput.}, 49(2):284--317, 2020.

\bibitem{Chung92} F.~Chung, P.~Diaconis, and R.~Graham.
\newblock Universal cycles for combinatorial structures.
\newblock {\em Discrete Mathematics}, {\bf 110}:43--59, 1992.

\bibitem{Dai93} Z.-D. Dai, K.~M. Martin, M.~J.~B. Robshaw, and P.~R. Wild.
\newblock Orientable sequences.
\newblock In M.~J. Ganley, editor, {\em Cryptography and Coding III}, pages
  97--115. Oxford University Press, Oxford, 1993.

\bibitem{Fredricksen82} H.~Fredricksen.
\newblock A survey of full length nonlinear shift register cycle algorithms.
\newblock {\em {SIAM} Review}, {\bf 24}:195--221, 1982.

\bibitem{Gabric24} D.~Gabric and J.~Sawada.
\newblock Efficient construction of long orientable sequences, 2024.
\newblock Available at \url{https://arxiv.org/abs/2401.14341}.

\bibitem{Golomb67} S.~W. Golomb.
\newblock {\em Shift register sequences}.
\newblock Holden-Day, San Francisco, 1967.

\bibitem{Jackson09} B.~Jackson, B.~Stevens, and G.~Hurlbert.
\newblock Research problems on {Gray} codes and universal cycles.
\newblock {\em Discret. Math.}, 309(17):5341--5348, 2009.

\bibitem{Lempel70} A.~Lempel.
\newblock On a homomorphism of the de {B}ruijn graph and its application to the
  design of feedback shift registers.
\newblock {\em {IEEE} Transactions on Computers}, {\bf C-19}:1204--1209, 1970.

\bibitem{Mitchell96} C.~J. Mitchell, T.~Etzion, and K.~G. Paterson.
\newblock A method for constructing decodable de {B}ruijn sequences.
\newblock {\em IEEE Transactions on Information Theory}, 42:1472--1478, 1996.

\bibitem{Mitchell22} C.~J. Mitchell and P.~R. Wild.
\newblock Constructing orientable sequences.
\newblock {\em IEEE Transactions on Information Theory}, 68:4782--4789, 2022.

\bibitem{Nellore22} A.~Nellore and R.~A. Ward.
\newblock Arbitrary-length analogs to de {Bruijn} sequences.
\newblock In H.~Bannai and J.~Holub, editors, {\em 33rd Annual Symposium on
  Combinatorial Pattern Matching, {CPM} 2022, June 27--29, 2022, Prague, Czech
  Republic}, volume 223 of {\em LIPIcs}, pages 9:1--9:20. Schloss Dagstuhl ---
  Leibniz-Zentrum f{\"{u}}r Informatik, 2022.
\newblock Available at \url{https://doi.org/10.4230/LIPIcs.CPM.2022.9}.

\bibitem{C35} E.~M. Petriu.
\newblock Absolute-type position transducers using a pseudorandom encoding.
\newblock {\em {IEEE} Transactions on Instrumentation and Measurement}, {\bf
  IM-36}:950--955, 1987.

\bibitem{Ronse84} C.~Ronse.
\newblock {\em Feedback Shift Registers}, volume 169 of {\em Lecture Notes in
  Computer Science}.
\newblock Springer, 1984.

\bibitem{Sawada23} J.~Sawada, J.~Sears, A.~Trautrim, and A.~Williams.
\newblock Concatenation trees: {A} framework for efficient universal cycle and
  de {Bruijn} sequence constructions, 2023.
\newblock Available at \url{https://arxiv.org/abs/2308.12405}.

\bibitem{Szentandrasi12} I.~Szentandr{\'{a}}si, M.~Zachari{\'{a}}s, J.~Havel,
    A.~Herout,
  M.~Dubsk{\'{a}}, and R.~Kajan.
\newblock Uniform marker fields: {Camera} localization by orientable {De
  Bruijn} tori.
\newblock In {\em 11th {IEEE} International Symposium on Mixed and Augmented
  Reality, {ISMAR} 2012, Atlanta, GA, USA, November 5--8, 2012}, pages
  319--320. {IEEE} Computer Society, 2012.

\end{thebibliography}

\end{document}